\documentclass[11pt]{amsart}

\usepackage{amsmath}
\usepackage{mathtools}
\usepackage{fullpage}
\usepackage{xspace}
\usepackage[psamsfonts]{amssymb}
\usepackage[latin1]{inputenc}
\usepackage{graphicx,color}
\usepackage{hyperref}
\usepackage{graphicx}

\usepackage{amsmath}%
\usepackage{amsthm}%
\usepackage{amscd}
\usepackage{amsfonts}%
\usepackage{amssymb}%
\usepackage{graphicx}
\usepackage{stmaryrd}
\usepackage{mathrsfs}
\usepackage{tikz}
\usetikzlibrary{matrix,arrows}
\usepackage{tikz-cd}
\usepackage[all]{xy}
\usepackage{bbm} 
\usepackage{longtable}

\usepackage{lineno} 

\newtheorem{theorem}{Theorem}[section]

\newtheorem{definition}{Definition}

\newtheorem{proposition}[theorem]{Proposition}
\theoremstyle{remark}
\newtheorem{remark}[theorem]{Remark}

\numberwithin{equation}{section}

\newcommand{\afrak}{\mathfrak{a}}

\newcommand{\mfrak}{\mathfrak{m}}
\newcommand{\nfrak}{\mathfrak{n}}
\newcommand{\nnn}{\mathfrak{n}_0}
\newcommand{\mcop}{\mfrak_{\mathfrak{n}_0}}

\newcommand{\pfrak}{\mathfrak{p}}
\newcommand{\qfrak}{\mathfrak{q}}

\newcommand{\C}{\mathbb{C}}

\newcommand{\F}{\mathbb{F}}

\newcommand{\N}{\mathbb{N}}

\newcommand{\Z}{\mathbb{Z}}

\newcommand{\Dcali}{\mathcal{D}}

\newcommand{\Icali}{\mathcal{I}}

\newcommand{\Ocali}{\mathcal{O}}

\newcommand{\Tcali}{\mathcal{T}}

\newcommand{\Ycali}{\mathcal{Y}}

\newcommand{\End}{\mathrm{End}}

\newcommand{\Pic}{\mathrm{Pic}}

\newcommand{\GL}{\mathrm{GL}}
\newcommand{\PGL}{\mathrm{PGL}}

\newcommand{\Nr}{\mathrm{Nr}}

\newcommand{\ds}{\displaystyle}

\newcommand{\Tr}{\text{Tr}}

\newcommand{\ebar}{\overline{e}}
\newcommand{\Hbar}{\underline{\mathrm{H}}}

\newenvironment{customtheorem}[1]{%
  \renewcommand\thetheorem{#1}
  \theorem
}{\endtheorem}

\newcommand\rquot[2]{
  \mathchoice
  {
    \text{\raise0.5ex\hbox{$#1$}\big/\lower0.5ex\hbox{$#2$}}%
  }
  {
    #1\,/\,#2
  }
  {
    #1\,/\,#2
  }
  {
    #1\,/\,#2
  }
}

\newcommand\lrquot[3]{
  \mathchoice
  {
    \text{\lower0.5ex\hbox{$#1$}\big\backslash\raise0.5ex\hbox{$#2$\!}\big/
      \lower0.5ex\hbox{\!\!$#3$}}%
  }
  {
    #1\,\backslash\,#2\,/\,#3
  }
  {
    #1\,\backslash\,#2\,/\,#3
  }
  {
    #1\,\backslash\,#2\,/\,#3
  }
}

\newcommand\lquot[2]{
  \mathchoice
  {
    \text{\lower0.5ex\hbox{$#1$}\big\backslash\raise0.5ex\hbox{$#2$}}%
  }
  {
    #1\,\backslash\,#2
  }
  {
    #1\,\backslash\,#2
  }
  {
    #1\,\backslash\,#2
  }
}

\usepackage{fontawesome}

  \DeclareFontFamily{U}{wncy}{}
    \DeclareFontShape{U}{wncy}{m}{n}{<->wncyr10}{}
    \DeclareSymbolFont{mcy}{U}{wncy}{m}{n}
    \DeclareMathSymbol{\Sha}{\mathord}{mcy}{"58}

\begin{document}
\title{Equidistribution of Hecke orbits on the Picard group of definite Shimura curves}
\author{Matias Alvarado and Patricio P\'erez-Pi\~na}
\address{ Departamento de Matem\'aticas,
Pontificia Universidad Cat\'olica de Chile.
Facultad de Matem\'aticas,
4860 Av.\ Vicu\~na Mackenna,
Macul, RM, Chile}
\email[M. Alvarado]{mnalvarado1@uc.cl }
\email[P. P\'erez-Pi\~na]{paperez15@uc.cl }%
\date{\today}
\maketitle
\begin{abstract}
We prove an equidistribution result about Hecke orbits on the Picard group of Shimura curves coming from definite quaternion algebras over function fields. In particular, we show the equidistribution of Hecke orbits of supersingular Drinfeld modules of rank 2. Our approach is via the automorphic method, using bounds for coefficients of cuspidal automorphic forms of Drinfeld type as the main tool.
\end{abstract}



\section{Introduction}

Let $p$ be an odd prime number and let $q$ be a power of $p$. Denote by $k$ the field $\F_q(t)$ and let $A$ be the polynomial ring $\F_q[t]$. The infinite place $\infty$ corresponds to the valuation $v_\infty$ given by $-\deg$. Consider $\mathcal{D}$ a definite quaternion algebra over $k$ i.e. $\mathcal{D}$ ramifies at the place $\infty$ of $k$. Let $\nfrak_0$ be the product of the finite primes at which $\Dcali$ is ramified. Fix a maximal $A$-order $R$ in $\mathcal{D}$. The Shimura curve $X_{\nfrak_0}$ as defined in \cite{weiyu2011theta} is constructed as follows. Let $Y$ be the genus $0$ curve defined over $k$ whose points over a $k$-algebra $S$ are given by \[Y(S)=\{x\in \mathcal{D}\otimes S\mid x\neq0, \Tr(x)=\Nr(x)=0\}/S^\times,\] where the action of $S^\times$ is by right multiplication. The group $\mathcal{D}^\times$ acts on $Y$ on the right by conjugation. Let $\widehat{k}$ be the ring of finite adeles of $k$ and let $\widehat{A}$ the closure of $A$ in $\widehat{k}$. Naturally we denote $\widehat{\mathcal{D}}=\mathcal{D}\otimes\widehat{k}$ and $\widehat{R}=R\otimes\widehat{A}$. The Shimura curve $X_{\nfrak_0}$ is defined as the double quotient space \[X_{\nfrak_0}\coloneq(\widehat{R}^\times\backslash\widehat{D}^\times\times Y)/\mathcal{D}^\times.\]

The double quotient space \begin{equation}\label{leftideals}\widehat{R}^\times\backslash \widehat{\mathcal{D}}^\times /\mathcal{D}^\times\end{equation} is finite and by a local-global principle, it is in bijection with the finitely many equivalence classes of left ideals of $R$. Let $n$ be this quantity and consider $\{I_i\}_{i=1}^n$ a set of representatives for the left ideals of $R$. Let $R_i$ be the maximal right $A$-order in $\mathcal{D}$ defined by $R_i=\{x\in \mathcal{D}\mid I_ix\subseteq I_i\}$. If $\{g_i\}_{i=1}^n$ denotes a set of representatives for \eqref{leftideals}, then $R_i=\mathcal{D}\cap g_i^{-1}\widehat{R}g_i$.

The Shimura curve is the finite union of the $n$ genus $0$ curves defined over $k$ given by $X_i=Y/R_i^\times$. If $e_i$ denotes the class of degree $1$ in $\Pic(X_{\nfrak_0})$ corresponding to the component $X_i$, then we have \[\Pic(X_{\nfrak_0})\cong \Z e_1\oplus...\oplus\Z e_n.\]

Let $\mathcal{B}=\{e_1,...,e_n\}$. The degree of an element $e=\sum a_ie_i$ in $\Pic(X_{\nfrak_0})$ is $\deg e=\sum a_i$. If $\deg e\neq 0$, the divisor $e$ defines a signed probability measure in $\mathcal{B}$ given by $\delta_e=\frac{1}{\deg e}\sum_{i=1}^n a_i\delta_{e_i}$, where $\delta_{e_i}$ is the Dirac measure supported on $e_i$. 

When $f:\mathcal{B}\to \C$ is a function, we denoted by $\delta_{e}(f)$ the integral of $f$ with respect to the measure $\delta_e$. That is, 
$$\delta_e(f):=\int fd\delta_e=\dfrac{1}{\deg e} \sum_{j=1}^{n}a_jf(e_j).$$

Let $w_i=\#(R_i^\times)/(q-1)$ and define $e^*=\sum w_i^{-1}e_i$. We denote by $\mu$ the probability measure $\delta_{e^*}$.

Let $\mfrak$ be an ideal in $A$ let $t_\mfrak$ be the endomorphism on $\Pic(X_{\nfrak_0})$ coming from the respective Hecke correspondance in $X_{\nfrak_0}$ as defined in \cite{weiyu2017central} section 1.2. We denote by $\mfrak_{\nnn}$ the unique ideal coprime to $\nnn$ such that $\mfrak=\mfrak_{\nnn}\prod_{\pfrak\mid\nnn}\pfrak^{v_\pfrak(\mfrak)}$. Here $v_\pfrak$ is the valuation associated to the prime ideal $\pfrak$.

\begin{customtheorem}{A}[]\label{mainthm1} Let $1\leq i\leq n$. When ordered by $\deg\mcop$, the collection of Hecke orbits $\{t_\mfrak e_i\}$ becomes equidistributed on $\mathcal{B}$ following the measure $\mu$. In other words, for every function $f\colon \mathcal{B}\to\C$, we have that \[\delta_{t_\mfrak e_i}(f)\to\delta_{e^*}(f)\mbox{ as }\deg\mcop\to\infty.\]
\end{customtheorem}

We have the following corollary in the case that $\nfrak_0=\pfrak$ is a prime. Let $\F_\pfrak$ be the residue field $A/\pfrak$. Recall that a Drinfeld module $\phi$ of rank 2 over $\overline{\F_\pfrak}$ is said to be supersingular if its endomorphism ring $\End(\phi)$ is an order in the definite quaternion algebra ramified at $\pfrak$. Denote by ${D_\pfrak^{ss}}$ the set of supersingular Drinfeld modules of rank $2$ over $\overline{\F_\pfrak}$. By Theorem 2.6 in \cite{papikian2005variation}, $\mathcal{B}$ is in natural bijection with the set ${D_\pfrak^{ss}}$. It is given by sending $e_i$ to the unique supersingular Drinfeld module $\phi_i$ such that $\End(\phi_i)\simeq R_i.$ Let $\Phi=\sum w_i^{-1}\phi_i$ be the divisor over $D_\pfrak^{ss}$ corresponding to $e^*$ via this identification.

\begin{customtheorem}{B}[]\label{mainthm2}
Let $\phi$ be an element in $D^{ss}_\pfrak$. When ordered by $\deg\mfrak_\pfrak$, the collection of Hecke orbits $\{T_\mfrak \phi\}$ becomes equidistributed on $D^{ss}_\pfrak$ following the measure $\delta_{\Phi}$. In other words, for every function $f\colon D_\pfrak^{ss}\to\C$, we have that \[\delta_{T_\mfrak \phi}(f)\to\delta_{\Phi}(f)=\frac{q^2-1}{q^{\deg\pfrak}-1}\sum_{i=1}^n\frac{f(\phi_i)}{w_i}\mbox{ as }\deg\mfrak_\pfrak\to\infty.\]
\end{customtheorem}

\begin{proof}
Use the previous bijection and apply Theorem \ref{mainthm1} with $\Dcali$ the definite quaternion algebra ramified at $\pfrak$. In this case, the mass formula for supersingular Drinfeld modules (\cite{gekeler83} Satz 5.9(iii)) implies that $$\deg(\Phi)=\sum_{i=1}^n \dfrac{1}{w_i}=\dfrac{q^{\deg \pfrak}-1}{q^2-1}.$$
\end{proof}

Our approach to prove Theorem \ref{mainthm1} is the same as in \cite{menares2012equidistribution}. We relate the coefficients of $t_\mfrak e_i$ and $e^*$ with Fourier coefficients of an appropriate automorphic form of Drinfeld type, and then use Deligne-Ramanujan bounds to show the result.

\subsection*{Overview of the article}
In section 2 we explain the strategy to prove our main theorem. Namely, we recall the Brandt matrices and their relation with Hecke operators to finally explain their role in the proof of Theorem \ref{mainthm1}. We show how the $*$-weak convergence is a consequence of the convergence to zero of a sequence involving the coefficients of the Brandt matrices.
In section 3 we introduce the automorphic forms of Drinfeld type, which can be seen as an analogue of classical modular forms. We recall some properties, the action of Hecke operators, their Fourier expansion and Deligne-Ramanujan bounds. We end the section by recalling a family of automorphic forms that allows us to decompose the space of forms in a way suitable to our purposes.
Finally, in section 4 we prove Theorem \ref{mainthm1}.

\section*{Acknowledgements}
The first author was supported by ANID Doctorado Nacional 21200910. The second author
was supported by ANID Doctorado Nacional 21200911. We thank Ricardo Menares for valuable comments and for having read this manuscript. We are also grateful to
Mihran Papikian for answering questions about automorphic forms.

\section{Strategy of the proof}\label{secpreliminaries}

The main tool in order to describe our strategy are Brandt matrices for which we refer to \cite{weiyu2011theta}.

Let $\Nr(I_i)$ be the fractional ideal of $A$ generated by the elements $N(b)$ with $b\in I_i$. Set $N_{ij}\coloneq \Nr(I_i)\Nr(I_j)^{-1}$. Define \[B_{ij}(\mfrak)=\frac{\#\{b\in I_iI_j^{-1}\mid \Nr(b)N_{ij}^{-1}=\mfrak\}}{\#(R_j^\times)}\]

The matrix $B(\mfrak)=(B_{ij}(\mfrak))$ belongs to $M_n(\Z)$ and it is called the $\mfrak$-th Brandt matrix.

\begin{proposition}\label{heckeformula}
    The following formula holds: \[t_\mfrak e_i=\sum_{j=1}^n B_{ij}(\mfrak)e_j\]
\end{proposition}
\begin{proof}
    See Proposition 1.5 in \cite{weiyu2011theta}.
\end{proof}

Recall that we are interested in the measure associated to the divisor 
$$e^*=\sum_{j=1}^{n}\dfrac{1}{w_j}e_j.$$

The degree of $t_\mfrak e_i$ is equal to $\sum_{j=1}^{n}B_{ij}(\mfrak)$. This quantity is related to the arithmetic function $\sigma_{\nfrak_0}$, which is defined by

\begin{align*}
\sigma_{\nfrak_0} \colon \F_q[t] & \longrightarrow \N \\
a&\longmapsto \sum_{m \text{ monic }\atop m|a, \ (\nfrak_0,m)=1} q^{\deg m}
\end{align*}

\begin{proposition}\label{degtm}
The degree of $\deg t_\mfrak e_i$ does not depends on $i$ and it equals $\sigma_{\nfrak_0}(\mfrak).$
\end{proposition}
\begin{proof}
    See the remark right after Proposition 2.4 in \cite{weiyu2011theta}.
    
\end{proof}
Theorem \ref{mainthm1} establishes that $\delta_{T_\mfrak e_i}(f)\to \delta_{e^*}(f)$ for every function $f\colon \mathcal{B}\to\C$. Using Propositions \ref{heckeformula} and \ref{degtm}, we have that
$$\delta_{T_\mfrak e_i}(f)=\dfrac{1}{\sigma_{\nfrak_0}(\mfrak)}\sum_{j=1}^{n}B_{ij}(\mfrak)f(e_j) \text{ and } \delta_{e^*}(f)=\dfrac{1}{\deg(e^*)}\sum_{j=1}^{n}\dfrac{1}{w_j}f(e_j).$$

Then

\begin{align*}
    \left|\delta_{T_\mfrak e_i}(f)-\delta_{e^*}(f)\right| &= \left| \sum_{j=1}^{n}\dfrac{B_{ij}(\mfrak)}{\sigma_{\nfrak_0}(\mfrak)}f(e_j)-\frac{1}{\deg(e^*)}\sum_{j=1}^{n}\dfrac{f(e_j)}{w_j}\right| \\
    &\leq \max_j\left\{f(e_j)\right\}\sum_{j=1}^{n} \left|\dfrac{B_{ij}(\mfrak)}{\sigma_{\nfrak_0}(\mfrak)}-\dfrac{1}{w_j\deg(e^*)} \right|
\end{align*}

We conclude that if $\left|\dfrac{B_{ij}(\mfrak)}{\sigma_{\nfrak_0}(\mfrak)}-\dfrac{1}{w_j\deg(e^*)} \right|\to 0$ for all $j$, then $\delta_{T_\mfrak e_i}\overset{\ast}{\rightharpoonup} \delta_{e^*}$. Henceforth we will study this difference and prove its convergence to $0.$ In order to do this, in the next section we proceed to recall the theory of automorphic forms of Drinfeld type which will be the main tool.

The next proposition we will be used in the course of proving Theorem \ref{mainthm1}.

\begin{proposition}\label{propBrandt} For every index $i$, there exists an index $k$ such that \[B_{ij}(\mfrak)/\sigma_{\nfrak_0}(\mfrak)=B_{kj}(\mcop)/\sigma_{\nfrak_0}(\mcop)\] for every index $j$.
\end{proposition}
\begin{proof}
This can be proven using the same argument given in \cite{menares2012equidistribution} at the beginning of section 1.2. The argument relies on properties of the Brandt matrices that can be found in Proposition II.1.
\end{proof}

\section{Automorphic forms of Drinfeld type}

Let $k_\infty=\F_q((1/t))$ be the completion of $k$ at the infinite place, $\Ocali_\infty=\F_q[[1/t]]$ its ring of integers and $\pi_\infty$ a uniformizer. We begin this section by recalling some facts about the Bruhat-Tits tree of $\PGL_2(k_\infty)$. A complete exposition about it can be found in \cite{serretrees}.

Let $\Tcali$ be the Bruhat-Tits tree of $\PGL_2(k_\infty)$. We denote by $X(\Tcali)$ and $Y(\Tcali)$ its set of vertices and edges respectively. The automorphic forms of Drinfeld type will be functions on the set of oriented vertex of $\Tcali$, with values in $\C$. This functions play the role of modular forms in the classical setting. 

Define the Iwahori group $\Icali$ as 
\[\Icali=\left\{ \begin{pmatrix} a&b \\ c&d \end{pmatrix} \in \PGL_2(\Ocali_\infty)\left|c\in (\pi_\infty) \right.\right\}\] and denote the set of oriented edges of $\Tcali$ by $\Ycali(\Tcali)$. The group $\Icali$ is the stabilizer of a distingushed oriented edge under the transitive action of $\PGL_2(k_\infty)$ over $\Ycali(\Tcali)$. Therefore, there is an identification

\begin{equation}\label{edges}\rquot{\PGL_2(k_\infty)}{\Icali}\leftrightarrow \Ycali(\Tcali).\end{equation}



If $e$ is an oriented edge, then $\overline{e}$ denote the same edge with the inverted orientation. The functions $\Ycali(\Tcali)\to X(\Tcali)$ source and target are denoted by $s$ and $t$ respectively.

Let $\nfrak$ be an ideal of $A$. The subgroup $\Gamma_0(\nfrak)$ of $\GL_2(\F_q[t])$ is defined by

\begin{align*}
        \Gamma_0(\nfrak)=\left\{ \begin{pmatrix} a & b \\ c& d\end{pmatrix}\in \GL_2\left(\F_q[t]\right) |c\equiv 0 \text{ (mod }\nfrak) \right\}
    .\end{align*}

\begin{definition}
    An automophic form of Drinfeld type for the group $\Gamma_0(\nfrak)$ is a function $f \colon \Ycali(\Tcali) \to \C$ which satisfies
    \begin{itemize}
        \item[i)]$f(e)+f(\ebar)=0$ for any $e\in \Ycali(\Tcali).$
        \item[ii)]$\sum_{t(e)=v} f(e)=0$ for any $v\in X(\Tcali).$
        \item[iii)] $f(\gamma e)=f(e)$ for all $\gamma \in \Gamma_0(\nfrak)$.
    \end{itemize}

If additionally, $f$ has compact support modulo $\Gamma_0(\nfrak)$, we say that $f$ is cuspidal. 
\end{definition}

The space of automorphic forms of Drinfeld type for $\Gamma_0(\nfrak)$ is denoted by $\Hbar(\Tcali,\C)^{\Gamma_0(\nfrak)}$. The subspace of cuspidal forms is denoted by $\Hbar_!(\Tcali,\C)^{\Gamma_0(\nfrak)}$.

\begin{remark}
It is possible to define automoprhic forms of Drinfeld type for more general subgroups of $\GL_2(\F_q[t])$. In this work we will only use forms of level $\Gamma_0(\nfrak).$
\end{remark}

\subsection{Fourier analysis}

We follow the exposition given in \cite{gekelerweil} p.42. Since automorphic Drinfeld forms for $\Gamma_0(\nfrak)$ are left invariants by the action of the algebra $\begin{pmatrix} 1& \F_q[t] \\ 0&1\end{pmatrix}$, they admit a Fourier expansion as follows. Consider the additive character $\psi_\infty\colon k_\infty \to \C^\times$ given by 

$$\psi_\infty\left( \sum a_i\pi_\infty^i\right)=\exp\left(\dfrac{2\pi i}{p} \Tr_{\F_q/\F_p}(a_{1}) \right).$$

For any ideal $\afrak=\lambda \F_q[t]$ of $\F_q[t]$, there exists coefficients $c_f(\afrak)\in\C$ such that \begin{equation}\label{fourier1}f\begin{pmatrix} \pi_\infty^r&u \\0 & 1 \end{pmatrix}=\sum_{\substack{\lambda\in \F_q[t]\\\deg \lambda+2\leq r}}q^{\deg \lambda+2-r}c_f(\afrak)\psi_\infty(\lambda  u).\end{equation}

\begin{remark}
    The matrices $\begin{pmatrix}
        \pi_\infty^r&u\\0&1
    \end{pmatrix}$ form a system of representatives for $\mathcal{Y}(\mathcal{T})$ under the identification \eqref{edges} up to orientation.
\end{remark}

For $r$ in $\Z$ and $\lambda$ in $\F_q[t]$ define $c_f(r,\lambda)=q^{\deg \lambda+2-r}c_f(\afrak)$ if $\deg \lambda+2-r\leq 0$ and $c_f(r,\lambda)=0$ otherwise. Then \begin{equation}\label{fourier2}f\begin{pmatrix} \pi_\infty^r&u \\0 & 1 \end{pmatrix}=\sum_{\lambda \in \F_q[t]}c_f(r,\lambda)\psi_\infty(\lambda  u).\end{equation} Note that this is compatible with the Fourier expansion consider in \cite{weiyu2011theta} subsection 2.1. 

\begin{remark}
    We refer to \eqref{fourier1} or \eqref{fourier2} as the Fourier expansion of $f$. Likewise, we refer to both $c_f(\afrak)$ and $c_f(r,\lambda)$ as the Fourier coefficients of $f$.
\end{remark}

If an automorphic form of Drinfeld type is cuspidal, then its constant Fourier coefficients $c_f(r,0)$ vanish.


\subsection{Hecke operators.}

For each monic polynomial $\mfrak$, there is an action of the Hecke operator $T_\mfrak$ on $\Hbar(\Tcali,\C)^{\Gamma_0(\pfrak)}$ which preserve $\Hbar_!(\Tcali,\C)^{\Gamma_0(\pfrak)}$. An explicit description of this action and its relation with Fourier coefficients can be seen in section 2.2 in\cite{weiyu2011theta}. These operators satisfies the following properties
\begin{itemize}
    \item All $T_\mfrak$ commute
    \item $T_{\mfrak \mfrak'}=T_\mfrak T_{\mfrak'}$ for $\mfrak$ coprime to $\mfrak'.$
    \item $T_{\pfrak^\ell}=T_\pfrak^\ell.$ if $\pfrak\mid \nfrak$.
    \item If $\qfrak$ is coprime to $\nfrak$, then
    $$T_{\qfrak^{n+1}}=T_{\qfrak^{n}}T_\qfrak-q^{\deg \qfrak}T_{\qfrak^{n-1}}$$
\end{itemize}

As in the classical case, when working with congruence subgroups $\Gamma_0(\nfrak)$, one can define a space of oldform in $\Hbar_{!}(\Tcali,\C)^{\Gamma_0(\nfrak)}$ as those coming from lower level forms. The new-space is then defined as the orthogonal complement of the old-space with respect to the Petersson inner product. Details about these facts can be found in Section  2.2 in \cite{weiyu2011theta}. We denote the space of newforms by $\Hbar^{\mathrm{new}}_{!}(\Tcali,\C)^{\Gamma_0(\nfrak)}$.


Following section 4.8 in \cite{gekeler1996jacobians}, the space of newforms admits a basis consisting of eigenfunction for the action of all Hecke operators $T_\mfrak$ with $\mfrak$ coprime to $\nfrak$. Moreover, the eigenvalues of the operators $T_{\qfrak}$ ($\qfrak\nmid \nfrak $) satisfies the Ramanujan's bound. More precisely, let $f$ be an eigenfunction for every operator and denote by $\lambda_\qfrak(f)$ the eigenvalue at the operator $T_\qfrak$. Then, for every prime $\qfrak$, we have $|\lambda_\qfrak(f)|\leq 2q^{\deg \qfrak/2}$. This follows from the work of Drinfeld \cite{Drinfeld} on the properties of the Frobenius eigenvalues on the cohomology of modular curves. Using equation (3.13) in \cite{gekelerweil} and assuming the normalization $c_f((1))=1$, we get the equality
$\lambda_\qfrak(f)=q^{\deg \qfrak} c_f(\qfrak)$ from where 
$$|c_f(\qfrak)|\leq 2q^{-\deg \qfrak/2}.$$

Analogously to the classical case, using the linear recurrences obtained from \[\sum_{j\geq0} c_f(\qfrak^j)X^j=\begin{cases}(1-q^{-d}\lambda_\qfrak X+q^{-d}X^2)^{-1}&\mbox{if }\qfrak\nmid\nfrak\\ (1-q^{-d}\lambda_\qfrak X)^{-1}&\mbox{if }\qfrak\mid\nfrak\end{cases}\] (see equation (3.14) in \cite{gekelerweil}), one obtains that for any cuspidal form $f\in \Hbar_!(\Tcali,\C)^{\Gamma_0(\pfrak)}$ and $\varepsilon>0$, 
\begin{equation}\label{bounds}|c_f(\mfrak)|\ll_f \sigma_0(\mfrak)q^{-\deg\mfrak/2}\ll q^{(\varepsilon-1/2)\deg\mfrak},\end{equation} where $\sigma_0(\mfrak)$ denotes the number of divisor of $\mfrak$. We are using the fact that $\sigma_0(\mfrak)=o\left(||\mfrak||^{\varepsilon} \right)=o\left(q^{(\deg \mfrak)\varepsilon}\right)$. This  is analogous to the growth rate of the function $d(\cdot)$ which count the number of positive divisors of natural numbers.


\subsection{Eisenstein and $\Theta$-series}

We recall two auxiliary automorphic forms whose Fourier coefficients are related with the terms involved in the strategy to prove Theorem \ref{mainthm1}. 

Let $i,j$ be two indices in $\{1,...,n\}$. Following section 2.1.1 in \cite{weiyu2011theta}, there exists an automorphic form of Drinfeld type $\Theta_{ij}$ such that its level is $\Gamma_0(\nfrak_0)$ and its Fourier coefficients are given by the formula

\begin{equation*}
    c_{\Theta_{ij}}(r,\lambda) = \begin{cases}
              q^{-r}B_{ij}(\afrak) & \text{if } \lambda \F_q[t]=\afrak\neq 0,\\
              q^{-r}w_j & \text{if } \lambda=0.
          \end{cases}
\end{equation*}

Now, following \cite{weiyu2011theta} pag.739, we set the Eisenstein series $E_{\nfrak_0}$ as
$$E_{\nfrak_0}=\sum_{j=1}^{n}\Theta_{ij}.$$ This sum is independent of the choice of $i$. The Fourier coefficients of $E_{\nfrak_0}$ are 

\begin{equation*}
    c_{E_{\nfrak_0}}(r,\lambda) = \begin{cases}
              q^{-r}\sigma_{\nfrak_0}(\lambda) & \text{if } \lambda\F_q[t]=\afrak\neq 0,\\
              q^{-r}\ds\sum_{i=1}^{n}\dfrac{1}{w_i} & \text{if } \lambda=0.
          \end{cases}
\end{equation*}

Let $\Pic(X_{\nfrak_0})^\vee$ be the dual group $\mathrm{Hom}(\Pic(X_{\nfrak_0}),\Z)$ and let $e^{\vee}_i$ with $1\leq i\leq n$ be the dual basis of $\mathcal{B}$. Consider $\Phi\colon \Pic(X_{\nfrak_0})\times \Pic(X_{\nfrak_0})^\vee\to \Hbar(\Tcali,\C)^{\Gamma_0(\pfrak)}$ by \[\Phi(e,e')=q^2\sum_{ij} a_ia_j'\Theta_{ij},\] where $e=\sum a_ie_i$ and $e'=\sum a_j'e^\vee_j$.

\begin{theorem}\label{decomposition} Let $\mathbb{T}_\C$ be the Hecke alegbra over $\C$ generated by the $T_\mfrak$. Define $\Hbar^{\mathrm{new}}(\Tcali,\C)^{\Gamma_0(\nfrak_0)}\coloneq\Hbar^{\mathrm{new}}_!(\Tcali,\C)^{\Gamma_0(\nfrak_0)}\oplus \C E_{\nfrak_0}$. The map $\phi$ induces a Hecke-equivariant isomorphism \[(\Pic(X_{\nfrak_0})\otimes\C)\otimes_{\mathbb{T}_\C}(\Pic(X_{\nfrak_0})^\vee\otimes\C)\cong\Hbar^{\mathrm{new}}(\Tcali,\C)^{\Gamma_0(\nfrak_0)}\]
\end{theorem}
\begin{proof}
    See Theorem 2.6 in \cite{weiyu2011theta}.
\end{proof}

\section{proof of Theorem \ref{mainthm1}} 

Using Theorem \ref{decomposition}, we can write $\Theta_{ij}=g_{ij}+c_{ij}E_{\nfrak_0},$ with $g_{ij}\in \Hbar_!(\Tcali,\C)^{\Gamma_0(\nfrak)}$ and $c_{ij}\in\C$. We compute the constant $c_{ij}$ by comparing Fourier expansions. For $r$ a natural number, we have
$$c_{\Theta_{ij}}(r,0)=c_{g_{ij}}(r,0)+c_{ij}c_{E_{\nfrak_0}}(r,0).$$
Then, using that $g_{ij}$ is cuspidal, 
$$\dfrac{q^{-r}}{w_j}=c_{ij}q^{-r}\deg(e^*),$$ from where we conclude that $c_{ij}=\frac{1}{w_j\deg(e^*)}$ and so $g_{ij}=\Theta_{ij}-\frac{1}{w_j\deg(e^*)}E_{\nfrak_0}$. Analysing the Fourier coefficients once more, we get

$$c_{g_{ij}}(\mfrak)=q^{-(\deg\mfrak+2)}\left(B_{ij}(\mfrak)-\frac{\sigma_{\nfrak_0}(\mfrak)
}{w_j\deg(e^*)}\right).$$

Then, by \eqref{bounds} and Proposition \ref{propBrandt} we have that

\begin{align*}
\left| \dfrac{B_{ij}(\mfrak)}{\sigma_\pfrak(\mfrak)}-\frac{1}{w_j\deg(e^*)} \right|&=\left| \dfrac{B_{kj}(\mcop)}{\sigma_\pfrak(\mcop)}-\frac{1}{w_j\deg(e^*)} \right|\\
&=q^{\deg\mcop+2}\dfrac{|c_{g_{kj}}(\mcop)|}{|\sigma_{\nfrak_0}(\mcop)|}\\
&\ll_g q^{(\varepsilon+1/2)\deg\mcop}\dfrac{1}{q^{\deg \mcop}}\\
&=q^{(\varepsilon-1/2)\deg\mcop}.\end{align*} From this we conclude that as $\deg \mcop \to \infty$, the quantity $\dfrac{B_{ij}(\mcop)}{\sigma_\pfrak(\mcop)} \to \dfrac{1}{w_j\deg(e^*)}$ and this finishes the proof.

\bibliographystyle{amsalpha}
\bibliography{refs.bib}

\end{document}